\documentclass[11pt]{amsart}

\usepackage{amssymb}

\input xy

\xyoption{all}

\theoremstyle{plain}

\newtheorem{theorem}{Theorem}[section]

\newtheorem{corollary}[theorem]{Corollary}

\newtheorem{lemma}[theorem]{Lemma}

\newtheorem{proposition}[theorem]{Proposition}

\newtheorem{fact}[theorem]{Fact}

\newtheorem{claim}[theorem]{Claim}

\newtheorem{question}[theorem]{Question}

\theoremstyle{definition} 

\newtheorem{definition}[theorem]{Definition}

\newtheorem{example}[theorem]{Example}

\newtheorem{remark}[theorem]{Remark}

\newtheorem{notation}[theorem]{Notation}

\theoremstyle{remark}

\newcommand{\cf}{\operatorname{cf}}

\DeclareMathOperator\dom{dom}

\renewcommand{\phi}{\varphi}

\newcommand{\initial}\lessdot

\def\?{?\vadjust

{\vbox to 0pt{\vskip-7pt\hbox to 1.1\hsize{\hfill\huge ?!}}}}

\usepackage{verbatim}

 \def\nfork{\setbox0\hbox{$\bigcup$}%
 \setbox1=\hbox to \wd0{\hfil\vrule width 0.7pt depth 2pt height 7.5pt\hfil}%
 \wd1=0cm\relax\box1\box0}

\renewcommand{\epsilon}{\varepsilon}

\usepackage{enumerate}

\newcommand{\be}{\begin{enumerate}}
\newcommand{\ee}{\end{enumerate}}

\newcommand{\bd}{\begin{definition}}
\newcommand{\ed}{\end{definition}}

\begin{document}


\title{A Note on Edge Colorings and Trees}

\author{Adi Jarden}
\email{jardena@ariel.ac.il}
\address{Department of Mathematics.\\ Ariel University \\ Ariel, Israel}

\author{Ziv Shami}
\email{zivsh@ariel.ac.il}
\address{Department of Mathematics.\\ Ariel University \\ Ariel, Israel}

\begin{abstract}
We point out some connections between existence of homogenous sets for certain edge colorings and existence of branches in certain trees.

As a consequence, we get that any locally additive coloring (a notion introduced in the paper) of a cardinal $\kappa$ has a homogeneous set of size $\kappa$ provided that the number of colors, $\mu$ satisfies $\mu^+<\kappa$.

Another result is that an uncountable cardinal $\kappa$ is weakly compact if and only if $\kappa$ is regular, has the tree property and for each $\lambda,\mu<\kappa$ there exists $\kappa^*<\kappa$ such that every tree of height $\mu$ with $\lambda$ nodes has less than $\kappa^*$ branches.

\end{abstract}


\maketitle

\tableofcontents

\section{Introduction}

In \cite{sh42} Shelah studied additive colorings as a tool in the study of monadic theories. Here, we focus on combinatorics only. On the one hand, we study Ramsey theorems and the connection between them and the existence of branches in some trees. On the other hand, we find a new characterization of weakly compact cardinals.


In Section 2 we introduce the basic definitions and notations about the colorings we deal in this paper.

In Section 3 we point out a connection between additive colorings and locally-additive colorings.

In Section 4 we assign to each coloring $C$ a tree $T(C)$ and to each tree $T$ with an enumeration $e$, a locally additive coloring $C(T,e)$. The function $C \mapsto T(C)$ is not injective on the class of locally additive colorings. On the other hand, we prove that if $e$ covers $T$ in some sense then
the tree $T(C(T,e))$ is isomorphic to the history tree of $T$.

In Section \ref{section the tree property at successor cardinals}, we prove the following lemma:
for any uncountable regular cardinal, $\kappa$ and $\mu<\kappa$,
the existence of a homogeneous set of size $\kappa$ for every locally additive coloring from $[\kappa]^2$ to $\mu$ is equivalent to the property that every tree of height $\kappa$ whose levels have size $\leq \mu$ admits a $\kappa$-sized branch.

As a consequence, we get a characterization of the successor cardinals satisfying the tree property.

In Section \ref{section strongly inaccessibles} we find an equivalent definition for weakly compact cardinals using the tree power.

In Section 7 prove theorems about existence of large homogeneous sets. On the one hand, we slightly improve a theorem of Shelah about $*$-locally-additive colorings. On the other hand, we derive a theorem on locally-additive colorings from the lemma in Section 5.

\subsection*{Acknowledgments.}
We would like to thank Assaf Rinot for his significant contribution to this paper. The proof of $(2)\Rightarrow (1)$ in Lemma \ref{the main lemma} and  Theorem 5.2, as well as a full proof of Lemma 5.3 included in the paper, is due to Assaf Rinot. In particular, the definitions of the tree $ T(C)$ associated to a coloring $C$ (Definition \ref{definition T(C)}) and the definition of the locally additive coloring associated to an enumerated tree $C( T,e)$ (Definition \ref{definition C(T,e)}) are due to him.  Originally, we had proven only the direction $(1)\Rightarrow (2)$ in Lemma \ref{the main lemma} by a slightly different argument and deduced Theorem \ref{the only theorem}.

\section{Basic Definitions}
We commence by defining several classes of colorings: additive colorings, locally-additive colorings and *-locally-additive colorings. We generalize the well-known Erd\H{o}s and Rado arrow notation.

\begin{definition}
A \emph{coloring} (or edge coloring) of a set $A$ with $\mu$ colors is a function
\[
C:[A]^2 \to D
\]
for some set $D$ of cardinality $\mu$. A coloring of a set $A$ is a coloring of a set $A$ with $\mu$ colors for some $\mu$. A coloring is a coloring of some set $A$.  An element in $D$ is called \emph{a color}.

If $A$ is a linearly ordered set and $C$ is a coloring of $A$ then we write $C(a,b)$ instead of $C(\{a,b\})$ when $a<b$ are in $A$.
\end{definition}

\begin{definition}
Let $C$ be a coloring of a set $A$. A subset $B$ of $A$ is a called \emph{homogeneous for $C$} if $C \restriction [B]^2$ is constant.
\end{definition}

 \begin{definition}\label{definition additive}\cite[$\S$1]{sh42}
Let $(A,<)$ be a linear ordering and let $D$ be a set. A function $C:[A]^2\rightarrow D$
is said to be an additive coloring if there is a function
\[
s:D \times D \to D
\]
 such that if $a,b,c$ are in $A$ and
$a<b<c$ then:
\[
C(a,c)=s(C(a,b),C(b,c))
\]
\end{definition}

\begin{definition}\label{definition locally-additive}
Let $(A,<)$ be a linear ordering and let $D$ be a set. A function $C:[A]^2\to D$
is said to be a locally-additive coloring if for all $a<b<c<d$ we have:
\[
C(b,c)=C(b,d) \Rightarrow C(a,c)=C(a,d)
\]
\end{definition}

\begin{notation}\label{notation local additivity}
Let $(A,<)$ be a linear ordering, $D$ be a set and $C:[A]^2\to D$ be a locally additive coloring.
For $b$ in $A$, let $D_b$ be the set
of colors $t \in D$ such that $C(b,c)=t$ for some $c>b$. For every $a<b$ in $A$ we define the following function:
\begin{align*}
&u_{a,b}:D_b \to D\\
&u_{a,b}(C(b,c))=C(a,c)
\end{align*}
\end{notation}
Note that $u_{a,b}$ is well-defined  for each $a<b$ if and only if the coloring $C$ is locally-additive.

\begin{definition}
Let $\kappa$ be a cardinal. A coloring $C$ of $\kappa$
is said to be a locally-additive coloring if it is locally-additive with respect to the $\in$-order of $\kappa$. $C$ is a *-locally additive coloring  if $C$ is locally-additive with respect to the reverse order of $\kappa$.
\end{definition}

\begin{remark}
Every additive coloring is locally-additive.
\end{remark}

\begin{claim}\label{C^* is additive}
For every locally additive coloring $C$ on a set $A$, for all $a<b<c$ in $A$, we have:
\[
u_{a,b} \circ u_{b,c}=u_{a,c}
\]
\end{claim}

\begin{proof}
Let $a<b<c$. Each side of the equality is a function from \[D_c=\{t \in D:t=C(c,d) \text{ for some } d>c\}
\] into $\mu$. Let $t \in D_c$, namely, $C(c,d)=t$ for some $d>c$. On the one hand,
\[
u_{a,c}(t)=u_{a,c}(C(c,d))=C(a,d)
\]

But on the other hand,
\[
(u_{a,b} \circ u_{b,c})(t)=u_{a,b}(u_{b,c}(C(c,d)))=u_{a,b}(C(b,d))=C(a,d)
\]
\end{proof}

\begin{notation} Let $\mu$ be a cardinal. We introduce the following notations:
\be
\item $add$=the class of additive colorings.
\item $add_\mu$=the class of additive colorings with $\mu$ colors.
\item $add_{<\mu}$= the class of additive colorings with less than $\mu$ colors.
\item $la$= the class of locally-additive colorings.
\item *-$la$= the class of *-locally additive colorings on a regular cardinal.
\item Similarly, $la_\mu$, $la_{<\mu}$, *-$la_{\mu}$ and *-$la_{<\mu}$.
\ee
\end{notation}


We generalize the well-known Erd\H{o}s and Rado arrow notation:
\begin{definition}
Let $F_1,F_2$ be families of subsets of a fixed linearly ordered set. Let $\mathfrak{C}$ be a class of colorings.
\[
F_1 \rightarrow  (F_2)_\mathfrak{C}
\]
is the following principle:
For every set $D \in F_1$ and every coloring $C \in \mathfrak{C}$ of $D$, there is a subset $E$ of $D$
such that $E \in F_2$ and $C \restriction [E]^2$ is constant. We put $\kappa$ for $F_1$ when $F_1$ is the family of subsets of $\kappa$ of cardinality $\kappa$ and similarly for $F_2$. We put $\mu$ for $\mathfrak{C}$ when $\mathfrak{C}$ is the class of all colorings with $\mu$ colors.

\end{definition}
Special cases of $\mathfrak{C}$ may be $add$, $la$, *-$la$, $add_{<\aleph_0}$, $la_{\kappa}$,
 and more.

\section{Replacing Additivity by Local Additivity}
In this section we prove that every Ramsey theorem related to additive colorings can be translated to the context of locally additive colorings at one price: the number of colors is decreased (in the infinite case, it is decreased to its logarithm). As an application, we get a direct generalization of a theorem  of Shelah on additive colorings of dense linear orderings.

\begin{theorem}\label{generalizing add to loc less than mu colors}\label{theorem generalizing add to la}
Let $\kappa,\mu$ be cardinals such that $2^\mu+\aleph_0 \leq \kappa$ and let $F$ be a family of subsets of a linearly ordered set.
Assume that
\[
F \rightarrow (F)_{add_\kappa}
\]

Then
\[
F \rightarrow (F)_{la_\mu}
\]
(recall that $la_\mu$ is the class of locally additive colorings with $\mu$ colors).
\end{theorem}

\begin{proof}
Let $A \in F$ and $C:[A]^2 \to \mu$ be a locally additive coloring. For each $a<b$ in $A$ let $u_{a,b}$ be the function that is defined in Notation \ref{notation local additivity}.
Define:
\[
D^*=\bigcup_{D^- \subseteq \mu} \ ^{D^-} \mu
\]
(the set of functions of subsets of $\mu$ into $\mu$).  We define a new coloring
\begin{align*}
&C^*:[A]^2 \to D^* \\
&C^*(a,b)=u_{a,b} \text{ where } a<b
\end{align*}
By Claim \ref{C^* is additive}, $C^*$ is an additive coloring. Since $|D^*| \leq \kappa$, we can apply the assumption,
\[
F \rightarrow (F)_{add_\kappa}
\]
on the additive coloring $C^*$. So there is a subset $B$ of $A$ so that
$B \in F$ and $C^* \restriction [B]^2$ is constant. Let $f$ be the function, such that $C^*(a,b)$ equals $f$, whenever $a<b$ are in $B$.

Now for all $a<b<c$ are in $B$ the following holds:
\[
C(a,c)=C^*(a,b)(C(b,c))=f(C(b,c))
\]
So $C \restriction B$ is additive. We apply again the assumption,
\[
F \rightarrow (F)_{add_\kappa}
\]
to get a subset $E$ of $B$ such that
$E \in F$
and $C \restriction [E]^2$ is constant.
\end{proof}

\begin{corollary}\label{generalizing add to loc for strong limit cardinal}
Let $\mu$ be a strong limit cardinal and let $F$ be a family of subsets of a linearly ordered set.
Assume that
\[
F \rightarrow (F)_{add_{<\mu}}
\]

Then:
\[
F \rightarrow (F)_{la_{<\mu}}
\]

\end{corollary}

\begin{corollary}\label{generalizing add to loc for finite}
Let $F$ be a family of subsets of a linearly ordered set.
Assume that
\[
F \rightarrow (F)_{add_{<\aleph_0}}
\]

Then:
\[
F \rightarrow (F)_{la_{<\aleph_0}}
\]
\end{corollary}

In our notation, Shelah proved the following fact:
\begin{fact}\cite[Theorem 1.3]{sh42}\label{fact dense}
Let $I$ be a dense linear ordering and let $F$ be the family of subsets of $I$ that are dense in some interval of $I$. Then:
\[
F \rightarrow (F)_{add_{<\aleph_0}}
\]
\end{fact}

Theorem \ref{generalizing add to loc less than mu colors} and Fact \ref{fact dense} yield:
\begin{corollary}
Let $I$ be a dense linear ordering and let $F$ be the family of subsets of $I$ that are dense in some interval of $I$.  Then:
\[
F \rightarrow (F)_{la_{<\aleph_0}}
\]
\end{corollary}

\begin{proof}
By Fact \ref{fact dense} and Corollary \ref{generalizing add to loc for finite}
\end{proof}

\section{Trees and Colorings}

In this section we study two functions: one assigns to each coloring $C$ a tree $T(C)$ and the other assigns to each tree $T$ with an enumeration $\{t^i_\alpha:i<\mu_\alpha\}$ for each level of $T$, a locally additive coloring $C(T,e)$. The function $C \mapsto T(C)$ is not injective. By Theorem \ref{T isomorphic to T(C(T,e))}, if $e$ covers $T$ (see the definition below) then the tree $T(C(T,e))$ is isomorphic to its history tree, so it does not depend on $e$.
$\\$

Recall some basic notions and notations about trees. A \emph{tree} is a partial-order $(T,\leq_{T})$ such  that for each $t \in T$ the set $\{s \in T:s<_{T}t\}$ is well-ordered by $\leq_{T}$. A \emph{branch} in $T$ is a maximal linearly ordered subset of $T$. 
A \emph{cofinal branch} in $T$ is a branch whose order type is the height of $T$.

For a tree $T$ and an ordinal $\alpha$, let $Lev_\alpha(T)$ denote the $\alpha$-th level of $T$ and let $Lev_{<\alpha}(T)$ denote the subtree of $T$, consisting the nodes of level smaller than $\alpha$. We denote the height of $t \in T$ in $T$ by $ht(t,T)$ and $ht(T)$ is the first $\alpha$ so that $Lev_\alpha(T)=\emptyset$. A tree $T$ is said to be \emph{well-pruned} if $Lev_0(T)$ has only one element and we have:
\[
\forall x \in T \forall \alpha [ht(x,T)<\alpha<ht(T) \rightarrow \exists y \in Lev_\alpha(T)(x<_T y)]
\]
A \emph{leaf} in a tree is a maximal element in the tree. Clearly, every well-pruned tree of limit ordinal height has no leaves.

\subsection{Reconstructing a Tree from its History}

\begin{definition}
For a tree $T$ and $t \in T$, let $f_t:ht(t,T) \to T$ be the function such that $f_t(\alpha)$ is the unique $t'<_Tt$ in $Lev_\alpha(T)$. \emph{The history tree of $T$} is the tree
\[
F_T=(\{f_t:t \in T\},\subseteq)
\]
\end{definition}

\begin{definition}
For a tree $T$, let
$T^{succ}$ be the subtree of $T$ induced by the set
\[
\{t \in T:ht(t,T) \text{ is a successor ordinal}\}
\]
\end{definition}
Clearly,
\[
(F_T)^{succ}=(\{f_t:t \in T^{succ}\},\subseteq)
\]
(as $ht(f_t,F_T)=ht(t,T)$).

\begin{remark}\label{remark 3.3}
Let $T,S$ be trees. If $g:F_T \to F_S$ is an isomorphism then
$g \restriction (F_T)^{succ}$ is an isomorphism from $(F_T)^{succ}$ to $(F_S)^{succ}$.
\end{remark}

\begin{proposition}\label{proposition without leaves}
Let $T$ be a tree without leaves. Then $T$ is isomorphic to the tree $(F_T)^{succ}$.
\end{proposition}

\begin{proof}
Define a function $h:T \to (F_T)^{succ}$ as follows:
$h(t)$ is the unique function whose domain is $ht(t,T)+1$ such that $h(t)$ is $f_{t^*}$ for some $t^*>_T t$ with $ht(t^*,T)=ht(t,T)+1$. Since $T$ has no leaves, for each $t \in T$ we can find $t^*$. By the definition of a tree, $h(t)$ does not depend on the choice of $t^*$. Clearly, $h$ is an isomorphism.
\end{proof}

\begin{corollary}\label{corollary reconstruction a tree from its history}
Let $T,S$ be two trees without leaves. Then the trees $T,S$ are isomorphic if and only if the history trees $F_T,F_S$ are isomorphic.
\end{corollary}

\begin{proof}
For the non-trivial direction, assume that $F_T \cong F_S$. Then by Remark \ref{remark 3.3}, $(F_T)^{succ} \cong (F_S)^{succ}$ and so by Proposition \ref{proposition without leaves}, $T \cong S$.
\end{proof}

\subsection{The Tree of a Coloring}
The following definition assigns to each coloring a tree.
\begin{definition}\label{definition T(C)}\cite[Page 7]{ri:chain}
Let $\kappa$ and $\mu$ be cardinals and let $C:[\kappa]^2 \to \mu$ be a coloring. For every $\delta<\kappa$, consider the induced fiber map $C(\cdot,\delta):\delta\rightarrow\mu$ which satisfies $C(\cdot,\delta)(\alpha)=C(\alpha,\delta)$ for all $\alpha<\delta$.
Then, consider the following downward-closed subtree of $({}^{<\kappa}\mu,\subseteq)$:
$$  T(C):=\{ C(\cdot,\delta)\restriction\beta\mid \beta\le\delta<\kappa\}.
$$
\end{definition}

\begin{remark}
For each $f:\kappa \to 2$, define a coloring:
\begin{align*}
&C_f:[\kappa]^2 \to 2\\
&C_f(\alpha,\beta)=f(\alpha)
\end{align*}
Let $\mathcal{C}$ be the set of colorings $C_f$ where $f:\kappa \to 2$ is a function. Note that $\mathcal{C}$ is both locally-additive and $*$-locally-additive. All the trees $T(C)$ for $C \in \mathcal{C}$ are isomorphic to $(\kappa,\in)$.
\end{remark}

\begin{remark}
For each $f:\kappa \to 2$, define a coloring:
\begin{align*}
&C_f:[\kappa]^2 \to 2\\
&C_f(\alpha,\beta)=f(\beta)
\end{align*}
Let $\mathcal{C}$ be the set of colorings $C_f$ where $f:\kappa \to 2$ is a function which is not eventually constant. Note that $\mathcal{C}$ is both locally-additive and $*$-locally-additive. For each $C \in \mathcal{C}$, we have:
\[
T(C)=^{<\kappa}\{0\} \cup ^{<\kappa}\{1\}
\]
\end{remark}

\begin{question}
Find a characterization of the equivalence classes of locally additive colorings induced by the map $C \mapsto T(C)$.
\end{question}

\subsection{The Locally Additive Coloring of an Enumerated Tree}
\begin{definition}\label{definition enumeration}
Let $T$ be a $\kappa$-tree. For each $\alpha<\kappa$, let
\[
\langle t^i_\alpha:i<\mu_\alpha \rangle
\]
 be a bijective enumeration of $Lev_\alpha(T)$. A sequence
\[
e=\langle \langle t^i_\alpha:i<\mu_\alpha \rangle:\alpha<\kappa \rangle
\]
of enumerations of the levels of $T$
is called an \emph{enumeration} of $T$.
\end{definition}

\begin{definition}\label{definition C(T,e)}
Let $T$ be a $\kappa$-tree and let $e$ be an enumeration of $T$ as given by Definition \ref{definition enumeration}. Let $\mu$ be the supremum of $\{\mu_\alpha:\alpha<\kappa\}$. Let
\[
C(T,e):[\kappa]^2 \to \mu
\]
be the coloring defined by:
\[
C(T,e)(\alpha,\delta)=i
\]
if and only if $\alpha<\delta$ and $i<\mu$ is the unique ordinal such that $t^i_\alpha \leq_T t^0_\delta$.
\end{definition}

\begin{claim}\label{claim C(T,e) is locally additive}
For each $\kappa$-tree $T$ and an enumeration $e$ of $T$, the coloring $C(T,e)$ is locally additive.
\end{claim}
\begin{proof}
Denote $C=C(T,e)$. Let $\alpha<\beta<\gamma_1,\gamma_2<\kappa$. We have to show:
\[
C(\beta,\gamma_1)=C(\beta,\gamma_2) \Rightarrow C(\alpha,\gamma_1)=C(\alpha,\gamma_2)
\]
Let
\[
i=C(\beta,\gamma_1)=C(\beta,\gamma_2)
\]
So
\begin{align*}
&t^i_\beta \leq_T t^0_{\gamma_1},t^0_{\gamma_2}\\
&t^{C(\alpha,\gamma_1)}_\alpha \leq_T t^0_{\gamma_1}\\
&t^{C(\alpha,\gamma_2)}_\alpha \leq_T t^0_{\gamma_2}
\end{align*}
Since $\alpha<\beta$, we have:
\[
t^{C(\alpha,\gamma_1)}_\alpha,t^{C(\alpha,\gamma_2)}_\alpha \leq_T t^i_\beta
\]
Since $t^{C(\alpha,\gamma_1)}_\alpha,t^{C(\alpha,\gamma_2)}_\alpha$ are of the same level and comparable, we have:
\[
t^{C(\alpha,\gamma_1)}_\alpha=t^{C(\alpha,\gamma_2)}_\alpha
\]
By the injectivity of the enumeration of each level, we have:
\[
C(\alpha,\gamma_1)=C(\alpha,\gamma_2)
\]

\end{proof}

\begin{definition}
Let $e$ be an enumeration of a $\kappa$-tree $T$. We say that $e$ \emph{covers} $T$ if for every $t \in T$ for some $\delta<\kappa$, we have $t <_T t_\delta^0$.
\end{definition}

\begin{remark}
For every cardinal $\kappa$ and a well-pruned $\kappa$-tree $T$, there is an enumeration $e$ of $T$ that covers $T$.
\end{remark}

\begin{proof}
Fix a bijection $f:\kappa \to T$. We define $\delta_\gamma<\kappa$ and choose $t^0_{\delta_\gamma}$ by induction on $\gamma$ in the following way:
\begin{align*}
&\delta_\gamma=\max(ht(f(\gamma),T),\sup\{\delta_\beta:\beta<\gamma\})+1 \\
&f(\gamma) <_T t^0_{\delta_\gamma} \text{ and } ht(t^0_{\delta_\gamma},T)=\delta_\gamma
\end{align*}
  Since $\kappa$ is regular, $\delta_\gamma<\kappa$. Note that the function $\gamma \mapsto \delta_\gamma$ is increasing and so injective. For each $\delta \notin \{\delta_\gamma:\gamma<\kappa\}$ we choose $t^0_\delta$ arbitrarily. It remains to choose an enumeration $\langle t^i_\alpha:0<i<\mu_\alpha \rangle$ of $Lev_\alpha(T) \setminus \{t^0_\alpha\}$ for each $\alpha<\kappa$. Clearly, the enumeration $e$ covers $T$.
\end{proof}

\begin{theorem}\label{T isomorphic to T(C(T,e))}
Let $T$ be a $\kappa$-tree and let $e=\langle \langle t^i_\alpha:i<\mu_\alpha \rangle:\alpha<\kappa \rangle$ be an enumeration of $T$ as given by Definition \ref{definition enumeration}. If $e$ covers $T$ then $T(C(T,e))$ is isomorphic to the history tree of $T$, $\{f_t:t \in T\}$.
\end{theorem}

\begin{proof}
In order to simplify the notation, we write $C$ for $C(T,e)$. For $\delta<\kappa$ we stipulate $C(\delta,\delta)=0$.
We shall prove that the following function is a well-defined isomorphism:
\begin{align*}
&h:T(C) \to \{f_t:t \in T\}\\
&h(C(\cdot,\delta) \restriction \beta)=f_{t_\beta^{C(\beta,\delta)}} \text{ for every } \beta \leq \delta<\kappa
\end{align*}
We separate the proof into several claims. First note that the level of $t_\beta^{C(\beta,\delta)}$ is $\beta$, so the domain of the function $f_{t_\beta^{C(\beta,\delta)}}$ is $\beta$.
\begin{claim}\label{claim for theorem 4.21}
Let $\alpha<\beta \leq \delta<\kappa$ be levels  of $T$. Let $i<\mu$ and let $s=t^i_\alpha$. Then
\[
(C(\cdot,\delta) \restriction \beta)(\alpha) \text{ if and only if } f_{t^{C(\beta,\delta)}_\beta}(\alpha)=t^i_\alpha.
\]
\end{claim}

\begin{proof}
We prove that the following are equivalent:
\be
\item $(C(\cdot,\delta) \restriction \beta)(\alpha)=i$
\item $C(\alpha,\delta)=i$
\item $t_\alpha^i <_T t^0_\delta$
\item $s <_T t_\delta^0$
\item $s <_T t_\beta^{C(\beta,\delta)}$
\item $f_{t^{C(\beta,\delta)}_\beta}(\alpha)=t^i_\alpha$
\ee
The equivalence between every pair of consecutive clauses holds immediately by the definitions, except the equivalence between Clauses (4) and (5). Note that $t_\beta^{C(\beta,\delta)} \leq_T t^0_\delta$. Assume that Clause (5) holds. Then by the transitivity of $\leq_T$, Clause (4) holds. Conversely, assume that Clause (4) holds. So $s$ and $t_\beta^{C(\beta,\delta)}$ are comparable. As the level of $s$ is $\alpha$ and $\alpha<\beta$, Clause (5) holds.
\end{proof}

\begin{claim}\label{claim result of claim for theorem 4.21}
Let $\beta_1 \leq \delta_1$ and $\beta_2 \leq \delta_2$ be levels of $T$. The following are equivalent:
\be
\item
$C(\cdot,\delta_1) \restriction \beta_1=C(\cdot,\delta_2) \restriction \beta_2$ (so $\beta_1=\beta_2$).
\item
$f_{t^{C(\beta_1,\delta_1)}_{\beta_1}}=f_{t^{C(\beta_2,\delta_2)}_{\beta_2}}$
\ee
Hence, the function $h$ is well-defined and injective.
\end{claim}

\begin{proof}
The equivalence holds by claim \ref{claim for theorem 4.21} and the injectivity of each enumeration in $e$.
So by its definition, $h$ is well-defined and injective.
\end{proof}

\begin{claim}
The function $h$ is an order-homomorphism.
\end{claim}

\begin{proof}

For the first direction, suppose:
 $$C(\cdot,\delta_1) \restriction \beta_1 \subseteq C(\cdot,\delta_2) \restriction \beta_2$$
We should prove:
\[
f_{t_{\beta_1}^{C(\beta_1,\delta_1)}} \subseteq f_{t_{\beta_2}^{C(\beta_2,\delta_2)}}
\]
Let $\alpha<\beta_1$. Then:
\[
(C(\cdot,\delta_1) \restriction \beta_1)(\alpha)= (C(\cdot,\delta_2) \restriction \beta_2)(\alpha)
\]
So Claim \ref{claim for theorem 4.21} yields:
\[
f_{t_{\beta_1}^{C(\beta_1,\delta_1)}}(\alpha)= f_{t_{\beta_2}^{C(\beta_2,\delta_2)}}(\alpha)
\]
The converse is similar, using the injectivity of each enumeration in $e$.
\end{proof}

\begin{claim}
The function $h$ is surjective.
\end{claim}

\begin{proof}
Let $t \in T$. Since $e$ covers $T$, we can find $\delta$ such that $t <_T t_\delta^0$. Let $\alpha$ be the level of $t$. Then $t=t_\alpha^{C(\alpha,\delta)}$. So
\[
h(C(\cdot,\delta) \restriction \alpha)=f_{t_\alpha^{C(\alpha,\delta)}}=f_t
\]
\end{proof}
The proof of Theorem \ref{T isomorphic to T(C(T,e))} has completed.
\end{proof}

\begin{corollary}
Let $T_1,T_2$ be two $\kappa$-trees without leaves and for $i=1,2$ let $e_i$ be enumerations of $T_i$ that covers $T_i$. Then the trees $T_1,T_2$ are isomorphic if and only if the trees $T(C(T_1,e_1))$ and $T(C(T_2,e_2))$ are isomorphic.
\end{corollary}

\begin{proof}
The following conditions are equivalent:
\be
\item
The trees $T_1$ and $T_2$ are isomorphic;
\item
the history trees of $T_1$ and $T_2$  are isomorphic; and
\item
the trees $T(C(T_1,e_1))$ and $T(C(T_2,e_2))$ are isomorphic.
\ee
The equivalence between (2) and (3) is satisfied by Theorem \ref{T isomorphic to T(C(T,e))} and the equivalence between (1) and (2) is satisfied by Corollary \ref{corollary reconstruction a tree from its history}, as the trees has no leaves.
\end{proof}

\section{The Tree Property at Successor Cardinals}\label{section the tree property at successor cardinals}
In this section, we establish a connection between two properties:
\be
\item
existence of $\kappa$-sized branches in certain trees and
\item
existence of $\kappa$-sized homogeneous sets for locally additive colorings.
\ee
As a result, we get a characterization of the tree property at a successor cardinal.
$ $\\

Recall two equivalent definitions of a weakly compact cardinal:
\begin{fact} For every strongly inaccessible cardinal $\kappa$, the following are equivalent:
\begin{enumerate}
\item There exists no $\kappa$-Aronszajn tree;
\item For every $\lambda<\kappa$, every coloring $C:[\kappa]^2\rightarrow\lambda$ admits a homogeneous set of size $\kappa$.
\end{enumerate}
\end{fact}

Here, we shall get a similar characterization for successor cardinals.
\begin{theorem}\label{theorem 4.2}  For every infinite cardinal $\mu$, the following are equivalent:
\begin{enumerate}
\item There exists no $\mu^+$-Aronszajn tree;
\item $\mu^+ \rightarrow (\mu^+)_{la_\mu}$, namely: Every locally additive coloring $C:[\mu^+]^2\rightarrow\mu$ admits a homogeneous set of size $\mu^+$.
\end{enumerate}
\end{theorem}

The preceding theorem
is a particular case of the following lemma (where $\mu^+$ stands for $\kappa$).
\begin{lemma}\label{lemma 6}\label{the main lemma} Suppose that $\kappa$ is a regular uncountable cardinal and $\mu<\kappa$. The following are equivalent:
\begin{enumerate}
\item Every tree of height $\kappa$ whose levels have size $\le\mu$ admits a $\kappa$-sized branch.
\item $\kappa \rightarrow (\kappa)_{la_\mu}$, namely: Every locally additive coloring $C:[\kappa]^2\rightarrow\mu$ admits a homogeneous set of size $\kappa$.
\end{enumerate}
\end{lemma}
\begin{proof} $(1)\implies(2)$: Let $C:[\kappa]^2\rightarrow\mu$ be an arbitrary locally additive coloring.
\begin{claim} All the levels of $  T(C)$ have size $\le\mu$.
\end{claim}
\begin{proof} Suppose not. Pick $\beta<\kappa$ such that $|  T(C)\cap{}^\beta\mu|>\mu$.
Let $\langle \delta_i\mid i<\mu^+\rangle$
be a strictly increasing sequence of ordinals $<\kappa$ for which $i\mapsto C(\cdot,\delta_i)\restriction\beta$ is one-to-one over $\mu^+$.
Without loss of generality, $\delta_0>\beta$. Pick $i<j<\mu^+$ such that $C(\beta,\delta_i)=C(\beta,\delta_j)$.
As $C(\cdot,\delta_i)\restriction\beta\neq C(\cdot,\delta_j)\restriction\beta$, pick $\alpha<\beta$ such that $C(\alpha,\delta_i)\neq C(\alpha,\delta_j)$.
Then we arrive at the following contradiction:
$$C(\alpha,\delta_i)=u_{\alpha,\beta}(C(\beta,\delta_i))=u_{\alpha,\beta}(C(\beta,\delta_j))=C(\alpha,\delta_j).$$
\end{proof}
By (1), $  T(C)$ must admit a cofinal branch. Pick $b:\kappa\rightarrow\mu$ such that $\{ b\restriction\beta\mid \beta<\kappa\}\subseteq   T(C)$.
Fix a strictly increasing and continuous function $f:\kappa\rightarrow\kappa$ such that $f(0)=0$ and $b\restriction f(i)\subseteq C(\cdot,f(i+1))$ for all $i<\kappa$.
Pick $H\subseteq \{ f(i+1)\mid i<\kappa\}$ of size $\kappa$ on which $b$ is constant. Then $H$ is a $C$-homogeneous set of size $\kappa$.

$(\neg1)\implies(\neg2)$: Suppose that $ (T,<_T)$ is a tree of height $\kappa$ whose levels have size $\le\mu$, and $  T$ admits no $\kappa$-sized branch. We fix an enumeration $e$ of $T$ and consider the locally additive coloring $C(T,e)$ (given by Claim \ref{claim C(T,e) is locally additive}).

Towards a contradiction, suppose that $H$ is a $\kappa$-sized homogeneous set for $C(T,e)$, with value, say, $i$.
So for all $\alpha<\beta<\gamma$ from $H$ we have $t^i_\alpha\leq_T t_\gamma^0$ and $t^i_\beta\leq_T t_\gamma^0$,
and hence $t^i_\alpha\leq_T t^i_\beta$. So $\{ t^i_\alpha\mid \alpha\in H\}$ is a $\kappa$-sized chain in $  T$.
This is a contradiction.
\end{proof}

\section{Strongly Inaccessibles in the tree sense}\label{section strongly inaccessibles}
In this section, we prove that a regular cardinal that satisfies the tree property is strongly inaccessible if and only if it is strongly limit in the sense of the tree exponent.

\begin{definition}\cite{sh589}
For infinite cardinals $\mu \leq \lambda$, let $\mathfrak{T}_{\lambda,\mu}$ denote the set of trees $T$ of height $\mu$ with $\leq \lambda$ nodes. Let $Lim(T)$ denote the set of cofinal-branches of $T$. We define the tree exponent $\lambda^{\mu,tr}$ as follows:
\[
\lambda^{\mu,tr}=\sup\{|Lim(T)|:T \in \mathfrak{T}_{\lambda,\mu}\}
\]
\end{definition}

\begin{remark}\label{remark tree exponent leq exponent}
For every two infinite cardinals $\lambda$ and $\mu$ we have:
\[
\lambda^{\mu,tr} \leq \lambda^\mu
\]
\end{remark}

Recall, that a \emph{weakly compact} cardinal is a strongly inaccesible cardinal that satisfies the tree property.

\begin{fact}\label{fact 6.3}
Let $\kappa$ be an uncountable cardinal. The principle $\kappa \to (\kappa)_2$ holds if and only if $\kappa$ is weakly compact.
\end{fact}

By the following theorem, a cardinal satisfying the tree property is strongly inaccessible in the tree sense if and only if it is strongly inaccessible.
\begin{theorem}\label{theorem 5.4}
Let $\kappa$ be an uncountable regular cardinal that satisfies the tree property. The following are equivalent:
\begin{enumerate}
\item $\lambda^{\mu,tr}<\kappa$ for each $\lambda,\mu<\kappa$ and \item  $\kappa$ is strongly inaccessible.
\end{enumerate}
\end{theorem}

\begin{proof}
On the one hand, by Remark \ref{remark tree exponent leq exponent}, Clause $(2)$ implies Clause $(1)$. Conversely, assume that Clause $(1)$ holds. By Fact \ref{fact 6.3}, it is enough to prove that $\kappa \rightarrow (\kappa)_2$ holds.

Take a coloring $F:[\kappa]^2 \to 2$. We construct a tree $T=\{t_\alpha:\alpha<\kappa\}$ of functions of ordinals into $2$: We define $t_\alpha$ by induction on $\alpha<\kappa$. We define $t_\alpha \restriction \xi$ by induction on $\xi$. If $t_\alpha \restriction \xi$ is not in $\{t_\beta:\beta<\alpha \}$ then we define $t_\alpha=t_\alpha \restriction \xi$. Otherwise $t_\alpha \restriction \xi=t_\beta$ for some $\beta<\alpha$. In this case, we define $t_\alpha(\xi)=F(\beta,\alpha)$. Clearly, the domain of each $t_\alpha$ is an ordinal smaller than $\kappa$ and $T=\{t_\alpha:\alpha<\kappa\}$ is a tree. Note that $t_\alpha \neq t_\beta$ for all $\alpha \neq \beta$ and therefore if $t_\alpha<_T t_\beta$ then $\alpha<\beta$.

We prove by induction on $\alpha<\kappa$ that the cardinality of $Lev_\alpha(T)$ is smaller than $\kappa$. By the induction hypothesis and the regularity of $\kappa$, $|Lev_{<\alpha}(T)|<\kappa$. Therefore by Assumption (1), we have $|Lim(Lev_{<\alpha}(T))|<\kappa$. But the function on $Lev_\alpha(T)$ into $Lim(Lev_{<\alpha}(T))$ assigning $t_\beta$ to $\{t_\gamma \in T:t_\gamma <_T t_\beta\}$ is injective. So $|Lev_\alpha(T)|<\kappa$.

$T$ is a $\kappa$-tree. So by assumption, $T$ has a $\kappa$-branch, $B$. If $t_\alpha<_T t_\beta<_T t_\gamma$ are in $B$ then $t_\beta \restriction \xi=t_\gamma \restriction \xi=t_\alpha$ where $\xi=ht(t_\alpha,T)$. So $F(\alpha,\beta)=t_\beta(\xi)=t_\gamma(\xi)=F(\alpha,\gamma)$. For $i<2$, define:
\[
H_i=\{\alpha \in \kappa:t_\alpha \in B \text{ and }\forall \beta>\alpha [t_\beta \in B \Rightarrow F(\alpha,\beta)=i]\}
\]

Each ordinal in $\{\alpha \in \kappa:t_\alpha \in B\}$ is in $H_0$ or in $H_1$. So one of $H_0$ and $H_1$ has cardinality $\kappa$. But each $H_i$ is homogeneous, because $F({\alpha,\beta})=i$ holds whenever $\alpha<\beta$ are in $H_i$.
\end{proof}


\begin{corollary}
Let $\kappa$ be an uncountable cardinal. Then $\kappa$ is weakly compact if and only if $\kappa$ satisfies the following properties:
\be
\item $\kappa$ is regular,
\item $\kappa$ satisfies the tree property and
\item $\lambda^{\mu,tr}<\kappa$ for each $\lambda,\mu<\kappa$.
\ee
\end{corollary}

\begin{proof}
Assume that $\kappa$ is weakly compact. Clearly $\kappa$ satisfies Properties (1),(2) and is strongly inaccesible. By Remark \ref{remark tree exponent leq exponent} $\kappa$ satisfies Property (3).

Conversely, if $\kappa$ satisfies Properties (1)-(3) then by Theorem \ref{theorem 5.4}, $\kappa$ is strongly inaccesible, so weakly compact.
\end{proof}

\section{Ramsey Theorems for locally-additive colorings}

We present two theorems relating colorings of $\kappa$ by $\mu$ colors where $\mu^+<\kappa=cf(\kappa)$: while in Theorem \ref{the only theorem} the coloring is assumed to be locally-additive, in Theorem \ref{theorem shelah *-locally-additive} the coloring is assumed to be *-locally-additive. We observe that the result for $*$-local additivity coloring cannot be generalized to the case $\mu^+=\kappa$. For local additivity we do not know whether the generalization holds. Here is a summary of the results of this section:
\begin{align*}
&\kappa \rightarrow_{*-la} (\kappa)_\mu \text{ (Theorem \ref{theorem shelah *-locally-additive}) }\\
&\kappa \rightarrow_{la} (\kappa)_\mu \text{ (Theorem \ref{the only theorem}) } \\
&\mu^+ \nrightarrow_{*-la} (\mu^+)_\mu  \text{ (Proposition \ref{proposition mu^+<kappa can's be omitted})  }            \\
&\mu^+ \rightarrow_{la} (\mu^+)_\mu? \text{ (unknown) }
\end{align*}

\subsection{*-locally-additive colorings}
In this subsection we verify that Shelah's proof of (\cite[Theorem 1.1]{sh42}) works for *-locally additive colorings (rather than additive colorings) with few number of colors (rather than finite).
For the rest of this subsection, we assume:
\[
\mu<\kappa=\cf(\kappa) \text{ and }
C:[\kappa]^2 \to \mu \text{ is a *-locally additive coloring. }
\]

\begin{definition}
For $\beta,\gamma<\kappa$, we write $\beta \sim \gamma$ if for some $\alpha<\kappa$ with $\beta,\gamma<\alpha$, we have $C(\beta,\alpha)=C(\gamma,\alpha)$. In this case, we say that $\alpha$ witnesses that $\beta \sim \gamma$.
\end{definition}

\begin{remark}\label{remark 7.2}
By local additivity, if $\beta \neq \gamma$, $\alpha_1$ witnesses that $\beta \sim \gamma$ and $\alpha_1<\alpha_2$ then $\alpha_2$ witnesses that $\beta \sim \gamma$ as well. Hence, $\sim$ is an equivalence relation.
\end{remark}

\begin{lemma}\label{lemma number of equivalence classes}
If $\mu^+<\kappa$ then
the number of $\sim$-equivalence classes is $\mu$ at most.
\end{lemma}

\begin{proof}
Assume towards contradiction that there are $\mu^+$ equivalence classes at least. Let $\{\alpha_i:i<\mu^+\}$ be a system of representatives. Since $\mu^+<\kappa$, we can find an ordinal $\alpha^*<\kappa$ which is bigger than each $\alpha_i$. The set $\{C(\alpha_i,\alpha^*):i<\mu^+\}$ is included in $\mu$. So for some distinct $i,j$ we have $C(\alpha_i,\alpha^*)=C(\alpha_j,\alpha^*)$, so $\alpha_i \sim \alpha_j$, a contradiction.
\end{proof}

\begin{theorem}\label{theorem shelah *-locally-additive}
$\kappa \rightarrow (\kappa)_{*-la_\mu}$ holds whenever $\mu^+<\kappa=cf(\kappa)$.
\end{theorem}

\begin{proof}
 By lemma \ref{lemma number of equivalence classes}, some $\sim$-equivalence class, $A$, is unbounded. We construct by induction a strictly increasing sequence $\langle y_\alpha:\alpha \in \kappa\rangle$ of elements of $A$, such that whenever $\beta<\gamma<\alpha$ we have $C(y_\beta,y_\alpha)=C(y_\gamma,y_\alpha)$. Given $y_\beta$ for all $\beta<\alpha$ we choose $y_\alpha$ as follows: for every $\beta<\gamma$ there is $\epsilon_{\beta,\gamma}<\kappa$ such that $\epsilon_{\beta,\gamma}>y_{\beta,},y_{\gamma}$ and $C(y_\beta,\epsilon_{\beta,\gamma})=C(y_\gamma,\epsilon_{\beta,\gamma})$. By Remark \ref{remark 7.2}, we may choose $y_\alpha$ to be any element of $A$ greater than $\sup\{\epsilon_{\beta,\gamma}: \beta<\gamma<\alpha\}$.

It follows that for every $\beta<\gamma<\alpha$ in $\kappa$ we have $C(y_\beta,y_\alpha)=C(y_\gamma,y_\alpha)$. Consider the function $f:\kappa \setminus \{0\} \to \mu$ defined by $f(\alpha)=C(y_0,y_\alpha)$. As $\mu<\kappa=cf(\kappa)$, there is an unbounded subset $B$ of $\kappa \setminus \{0\}$ and $t \in \mu$ such that $C(y_0,y_\alpha)=t$ holds for each $\alpha \in B$. So $C(y_\beta,y_\alpha)=t$ whenver $\beta<\alpha$ are in $B$. Thus, the set  $\{y_\alpha:\alpha \in B \}$ is $C$-homogeneous of size $\kappa$.
\end{proof}

\begin{proposition}\label{proposition mu^+<kappa can's be omitted}
$\kappa \rightarrow (\kappa)_{*-la_\mu}$ does not hold if $\mu^+=\kappa$.
\end{proposition}

\begin{proof}
We present a coloring witnessing the negation of this principle. For each $\alpha<\kappa$, let $f_\alpha:\alpha \to \mu$ be an injection. For $\beta<\alpha<\kappa$, we define $C(\beta,\alpha)=f_\alpha(\beta)$. The coloring $C$ is $*$-locally additive vacuously, because for any $\beta<\gamma<\alpha<\kappa$, we have $C(\beta,\alpha) \neq C(\gamma,\alpha)$ . By the same reason, there is no homogeneous set of cardinality $3$.
\end{proof}

\subsection{locally-additive colorings}

If we replace in Proposition \ref{proposition mu^+<kappa can's be omitted} the property of *-local additivity by local additivity then we get the following open question:
\begin{question}
Does $\mu^+ \rightarrow  (\mu^+)_{la_\mu}$ hold?
\end{question}

\begin{theorem}\label{the only theorem}
Let $\kappa$ be a regular cardinal and let $\mu$ be a cardinal with
 $\mu^+<\kappa$. The following principle holds:

 \[
 \kappa \rightarrow  (\kappa)_{la_\mu}
 \]

\end{theorem}

\begin{proof}
Define $\lambda=\mu^+$. By Claim \ref{the exercise of kunen} (below), every tree of height $\kappa$ whose levels have size $\leq \mu$ is not $\kappa$-Aronszajn. So by Lemma \ref{the main lemma}, the principle $\kappa \rightarrow_{la} (\kappa)_\mu$ holds.
\end{proof}

Recall the following claim \cite[Proposition 7.9]{Large}.
\begin{claim}\label{the exercise of kunen}
Let $\kappa$ and $\lambda$ be cardinals such that $\kappa$ is regular and $\lambda<\kappa$. If $(T,<)$ is
a $\kappa$-tree whose levels have size $<\lambda$ then $(T,<)$ is not $\kappa$-Aronszajn.
\end{claim}

For completeness, we prove Claim \ref{the exercise of kunen}.
\begin{proof}
\emph{Case A:} $\lambda$ is regular. Without loss of generality $(T,<)$ is well-pruned
(it a well-known fact that every $\kappa$-tree has a well-pruned subtree which is a $\kappa$-tree and so there is no harm in assuming that $(T,<)$ is well-pruned).

\begin{claim}\label{argue}
For every ordinal $\alpha<\kappa$ with $cf(\alpha)=\lambda$ there is an ordinal $q(\alpha)<\alpha$, such that for every ordinal $\beta$ with $q(\alpha)<\beta<\alpha$, and any element $a \in Lev_{q(\alpha)}(T)$ there are no two distinct elements $b,c \in Lev_\beta(T)$ such that $a<b$ and $a<c$.
\end{claim}

\begin{proof}
Towards a contradiction assume that there exists an $\alpha$ with $cf(\alpha)=\lambda$ such that for every $q(\alpha)<\alpha$ there are $\beta,a,b,c$ such that $q(\alpha)<\beta<\alpha$, $a \in Lev_{q(\alpha)}(T)$, $b,c \in Lev_\beta(T)$, $a<_T b$ and $a <_T c$. We choose by induction on $i<\lambda$ an ordinal $\alpha_i<\alpha$ and an element $d_i \in Lev_\alpha(T)$ such that $d_i \restriction Lev_{\alpha_{i}}(T) \neq d_j \restriction Lev_{\alpha_{i}}(T)$ for each $j<i$. Suppose we have chosen $\alpha_j$ and $d_j$ for each $j<i$. Let $q=\sup_{j<i}\alpha_j$. Then for each $j_1,j_2<i$ the elements $d_{j_1} \restriction Lev_q(T)$ and $d_{j_2} \restriction Lev_q(T)$ are distinct. In order to find $\alpha_{i}$ and $d_i$, we apply our assumption (towards a contradiction), where we substitute $q$ in place of $q(\alpha)$. So there is an ordinal $\alpha_{i}$, an element $a \in Lev_{q}(T)$ and two elements $b,c \in Lev_{\alpha_{i}}(T)$ such that $a<b$ and $a<c$. By the induction hypothesis, there is at most one ordinal $j<i$ such that $d_j \restriction Lev_{q}(T)=a$. We choose $d_i \in Lev_{\alpha}(T)$ such that $d_i \restriction Lev_{q}(T)=a$ and if for some $j<i$, $d_j \restriction Lev_{q}(T)=a$ then $d_i \restriction Lev_{\alpha_{i}}(T) \neq d_j \restriction Lev_{\alpha_{i}}(T)$ (if $d_j \restriction Lev_{\alpha_{i}}(T)=b$ then $d_i \restriction Lev_{\alpha_{i}}(T)=c$ and vice versa). Hence, we can carry out the induction. The $d_i$'s are pairwise distinct. So $|Lev_\alpha(T)|$ is $\lambda$ at least, a contradiction. Claim \ref{argue} is proved.
\end{proof}

We now can apply Fodor's Lemma with the function from $\kappa$ to $\kappa$, $\alpha \mapsto q(\alpha)$. So for some stationary subset $S$ of $\kappa$ and some ordinal $\alpha^*<\kappa$, $q(\alpha)=\alpha^*$ for each $\alpha \in S$. Fix $a \in Lev_{\alpha^*}(T)$. The set $\{b \in Lev_{\geq \alpha^*}(T):a<_Tb\}$ is a chain of size $\kappa$ since $(T,<_T)$ is well-pruned. So $(T,<_T)$ is not $\kappa$-Aronszajn.

\emph{Case B:} $\lambda$ is singular so in partucular, $\lambda$ is a limit cardinal. For the sake of a contradiction, assume that $(T,<)$ is $\kappa$-Aronszajn. By Case A, for each $\mu<\lambda$ there is a level of size $\mu^+$ at least. Since $\lambda<cf(\kappa)=\kappa$, it follows that there is a level of cardinality $\lambda$ at least, contradicting our assumption. Claim \ref{the exercise of kunen} is proved.
\end{proof}

\bibliographystyle{amsplain}
\bibliography{lit}

\end{document}